%% file: S2.tex
\documentclass[runningheads,envcountsame,a4paper]{llncs} 

\usepackage{amssymb,amsmath}
\usepackage{wasysym}
\usepackage{graphicx}
\usepackage{epsfig}
\usepackage{latexsym}
\usepackage[english]{babel}
\usepackage[utf8]{inputenc}
\usepackage{csquotes}
\usepackage{algorithm} 
\usepackage{algpseudocode}
\usepackage{tikz}
\usetikzlibrary{arrows,shapes.geometric,positioning,calc}
\usepackage[colorinlistoftodos]{todonotes}
\usepackage{changepage}
\usepackage{subfig}
\usepackage{soul}
\usepackage{hyperref}
\usepackage{todonotes}
\usepackage{enumerate}

\newtheorem{thm}{Theorem}
\newtheorem{prop}[thm]{Proposition}
\newtheorem{cor}[thm]{Corollary}
\newtheorem{lem}[thm]{Lemma}
\newtheorem{conj}[thm]{Conjecture}

\newsavebox{\qedB}

\sbox{\qedB}{\setlength{\unitlength}{1mm}
 \begin{picture}(4,4)(0,0)
  \thinlines
  {\put(0,0){\framebox(2.83,2.83){}}}%
  {\put(1.17,1.17){\framebox(2.83,2.83){}}}%
  {\put(0,0){\framebox(4,4){}}}%
  {\put(1.17,1.17){{\rule{1ex}{1ex} }}}%
 \end{picture}}
 
\newcommand{\QEDB}{\ifmmode\def\next{\tag"\usebox{\qedB}"}%
 \else\let\next=\relax
 {\unskip\nobreak\hfil\penalty50
 \hskip2em\hbox{}\nobreak\hfil\usebox{\qedB}
 \parfillskip=0pt \finalhyphendemerits=0\penalty-100\bigskip}\fi\next}

\newcommand{\bprop}{\begin{prop}}
\newcommand{\eprop}{\end{prop}}
\newcommand{\bcor}{\begin{cor}}
\newcommand{\ecor}{\end{cor}}
\newcommand{\blem}{\begin{lem}}
\newcommand{\elem}{\end{lem}}

\title{The characteristic sequence of the integers that are the sum of two squares is not morphic}

\authorrunning{}
\tocauthor{}
\author{Shuo Li}

\institute{Laboratoire de Combinatoire et d'Informatique Mathématique,\\
Université du  Québec \`a Montréal,\\
CP 8888 Succ. Centre-ville, Montréal (QC) Canada H3C 3P8\\
}

\begin{document}

\maketitle


\input{Abstract}
\input{Introduction}

\input{morphic1}


\bibliographystyle{splncs03}
\bibliography{biblio}

\end{document}

%% file: Abstract.tex
\begin{abstract}
Let $(s_2(n))_{n\in \mathbb{N}}$ be a $0,1$-sequence such that, for any natural
number $n$, $s_2(n) = 1$ if and only if $n$ is a sum of two squares. In a
recent article, Tahay proved that the sequence $(s_2(n))_{n\in \mathbb{N}}$ is not $k$-automatic for any integer $k$, and asked if this sequence can be morphic. In this note, we give a negative answer to this question.
\end{abstract}

%% file: Introduction.tex
\section{Introduction}

For any integer $k\geq 1$, let us define the $0,1$-sequence $(s_k(n))_{n \in \mathbb{N}}$ such that, for any $n \in \mathbb{N}$, $s_k(n)=1$ if $n$ can be written as a sum of $k$ squares and $s_k(n)=0$ if not. From the definition, the sequence $(s_1(n))_{n \in \mathbb{N}}$ characterizes all squares in $\mathbb{N}$, while $(s_2(n))_{n \in \mathbb{N}}$ and $(s_3(n))_{n \in \mathbb{N}}$ characterize respectively the sequences A001481 and A004215 in OEIS~\cite{oeis}. It is conjectured by Bachet and proved by Lagrange in 1770 that every natural number can be written as a sum of four squares. Thus, $(s_k(n))_{n \in \mathbb{N}}$ is a constant sequence of $1$ for any $k \geq 4$. The automaticity and the morphicity of $(s_k(n))_{n \in \mathbb{N}}$ for $k=1,2,3$ are discussed in \cite{Tahay}, where the notion of automaticity and the morphicity will be given in Section \ref{sec:def}. The $2$-automaticity of $(s_3(n))_{n \in \mathbb{N}}$ is proven in~\cite{Cobham} and the non-automaticity of $(s_1(n))_{n \in \mathbb{N}}$ is proven in~\cite{Minsky}~\cite{Ritchie}. In \cite{Tahay}, Tahay proved that $(s_1(n))_{n \in \mathbb{N}}$ is morphic and $(s_2(n))_{n \in \mathbb{N}}$ is not automatic. Notice that the non-automaticity of $(s_2(n))_{n \in \mathbb{N}}$ can also be established from the results in~\cite{Jakub2020}. However, the morphicity of $(s_2(n))_{n \in \mathbb{N}}$ remains open.  

In this note, we first give a criterion of the non-morphicity of a sequence, then prove non-morphicity of $(s_2(n))_{n \in \mathbb{N}}$ as an application of this criterion, which answers a question of Tahay asked in~\cite{Tahay}. Moreover, using the same criterion, we prove the non-morphicity of $(s'_2(n))_{n \in \mathbb{N}}$, which characterize the integers that are the sum of two {\em non-zero} squares.

Notice that the sequence $(s_2(n))_{n \in \mathbb{N}}$ is {\em multiplicative} (see, for example \cite{Tahay}). By multiplicative, we mean, for any pair of coprime integers $(p,q)$, one has $s_2(p)s_2(q)=s_2(pq)$. The (non)-automaticity of (completely) multiplicative sequences has been intensively studied during the recent years, for example, \cite{HU201773} \cite{ALLOUCHE2018356} \cite{LI2020388} \cite{Oleksiy2} \cite{Oleksiy} \cite{Jakub2020} \cite{Mullner22}, while the (non)-morphicity of (completely) multiplicative sequences, from the author's knowledge, is rarely studied in the literature. This note gives an example to show the non-morphicity of a multiplicative sequence.

\section{Definitions and notation} 
\label{sec:def}
By \emph{alphabet}, we mean a finite set. Its elements are called \emph{letters}. Let $A^*$ denote the free monoid generated by $A$ under concatenations having neutral element the empty word $\varepsilon$. For any word $w=w_0w_1w_2\cdots w_{n-1} \in A^*$, the \emph{length} of $w$ is the integer $|w| = n$. The length of any \emph{infinite words} is infinite. For any letter $a \in A$, let $|w|_a$ denote the number of occurrences of the letter $a$ in $w$. A \emph{prefix} of a word $w=w_0w_1w_2\cdots$ is a finite string $w_0w_1w_2\cdots w_t$ such that $t \leq |w|$. For any pair of integers $i \leq j$, let $w[i,j]=w_iw_{i+1}\cdots w_j$.

Let $A$ and $B$ be two alphabets. A \emph{morphism} $\phi$ is a map $A^* \to B^*$ satisfying $\phi(xy)=\phi(x)\phi(y)$ for any pair of elements $x, y$ in $A^*$. The morphism $\phi$ is called \emph{$k$-uniform} if for all elements $a \in A$, $|\phi(a)|=k$, and it is called \emph{non-uniform} otherwise. A morphism $\phi$ is called a \emph{coding} function if it is $1$-uniform, and it is called \emph{non-erasing} if $\phi(a) \neq \varepsilon$ for all $a \in A$. For any positive integer $k$, by $\phi^k$, we mean the composition of $k$ times of the morphism $\phi$. Particularly, $\phi^0$ is the identity function.

Let $A$ be a finite alphabet, and let $(a_n)_{n \in \mathbb{N}}$ be an infinite sequence over $A$. It is called \emph{morphic} if there exists an alphabet $B$, a letter $ b \in B$, a word $w \in B^*$, a non-erasing morphism $\phi: B^{*} \to B^{*}$, and a coding function $\psi: B \to A$, such that $\phi(b)=bw$ and $$(a_n)_{n \in \mathbb{N}}= \lim_{i \to \infty}\psi(\phi^i(b)).$$ Moreover, for any positive integer $k \geq 2$, the sequence $(a_n)_{n \in \mathbb{N}}$ is called \emph{$k$-automatic} if $\phi$ is $k$-uniform, is called \emph{automatic} if $\phi$ is $k$-uniform for some integer $k \geq 2$, and it is called \emph{non-automatic} if it is not $k$-uniformly morphic for any integer $k \geq 2$.

%% file: morphic1.tex
\section{Non-morphicity of sequences}
\label{sec:proof}

The purpose of this section is to prove the following theorem: 

\begin{thm}
\label{morphic}
Let $(a_n)_{n \in \mathbb{N}}$ be a morphic sequence over the alphabet $A$. If there exists $a \in A$ and two positive numbers $0 < \gamma < 1$, $C >0$ such that $$|a[0,N]|_a \sim C\frac{N}{(\log(N))^{\gamma}},$$ then $(a_n)_{n \in \mathbb{N}}$ is not morphic.  
\end{thm}

\begin{prop}
\label{growth}
For any alphabet $A$ and any morphism $\phi: A^* \to A^*$, if there exists $a \in A$ and $w \in A^*$ such that $\phi(a)=aw$, then there exist real numbers $\alpha \geq 1$, $G>0$ and non-negative integers $l$ and $T \geq 1$ such that $$|\phi^{Tk}(a)|\sim G(Tk)^l\alpha^{Tk}.$$ Moreover, for any $b \in A$, there exist real numbers $\beta \leq \alpha$, $G'>0$, a non-negative integer $m$ and finitely many complex numbers $\beta_1, \beta_2, \cdots, \beta_s$ such that $|\beta_i|=\beta$ for all $i$ and that $$|\phi^{Tk}(a)|_b\sim  G'(Tk)^m\sum_{i=1}^s\beta_i^{Tk}.$$ Moreover, if $\beta=\alpha$, then there exists a unique real number $\beta_1=\beta$.  
\end{prop}

\begin{proof}
The first part a direct consequence of Remark 1 and Proposition 4.1 in \cite{Bell08}.\\ For the second part, from Corollary 8.2.3 in \cite{allouche_shallit_2003}, there exists a $d \times d$ matrix $M$, known as the incidence matrix of the morphism $\phi$ (see Definition 2 in \cite{Bell08} and Section 8.2 in \cite{allouche_shallit_2003}), such that for any $b \in A$, and any non-negative integer $m$, $$|\phi^k(a)|_b=\sum_{i,j}c_{i,j}m^{(k)}_{i,j},$$ where $M^k=(m^{(k)}_{i,j})_{1 \leq i,j \leq d}$ and $c_{i,j}$ are real numbers independent from $k$. Moreover, let $\{v_1,v_2, \cdots, v_r\}$ be the set of distinct eigenvalues of $M$, then $\alpha \geq |v_i|$ for any $i$. Let $J$ be the Jordan form of $M$, then, there are some matrix $S$ such that for any $k$, $M^k=S^{-1}J^kS$. Thus, for any pair of $1 \leq i,j \leq d$, there exist $r$ polynomials $p_{i,j,1}, p_{i,j,2}, \cdots, p_{i,j,r}$, such that $$m^{(k)}_{i,j}=\sum_{t=1}^r p_{i,j,r}(k)v_t^k.$$ Consequently, for any $b \in A$, there exist $r$ polynomials $p_{b,1}, p_{b,2}, \cdots, p_{b,r}$, such that $$|\phi^k(a)|_b=\sum_{t=1}^r p_{b,r}(k)v_t^k.$$ Thus, there exists a $\beta \in \mathbb{R}$ such that $|\phi^{kT}(a)|_b\sim  G'{kT}^m\sum_{i=1}^s\beta_i^{kT}$ and that $|\beta_i|=\beta$ for all $i$. For the uniqueness of $\beta_i$ in the case that $\beta=\alpha$, see Remark 1 in \cite{Bell08}.
\end{proof}

\begin{cor}
\label{sequence}
For any morphic sequence $(a_n)_{n \in \mathbb{N}}$ over the alphabet $A^*$, there exists an increasing sequence of integers $0 \leq N_1 \leq N_2 \leq \cdots \leq N_k \leq \cdots$, two real numbers $\alpha \geq 1$, $G>0$ and two non-negative integers $l$ and $T$ such that for any integer $k$, $N_k \sim G(Tk)^l\alpha^{Tk}$.\\
Moreover, for any $a \in A$, there exist real numbers $\beta \leq \alpha$, $G'>0$, a non-negative integer $m$ and finitely many complex numbers $\beta_1, \beta_2, \cdots, \beta_s$ such that $|\beta_i|=\beta$ for all $i$ and that $|a[0, N_k]|_a\sim  G'(Tk)^m\sum_{i=1}^s\beta_i^{Tk}.$ \\
Moreover, if $\beta=\alpha$, then there exists a unique real number $\beta_1=\beta$.  
\end{cor}

\begin{proof}
If $(a_n)_{n \in \mathbb{N}}$ is morphic, from the definition, there exists an alphabet $B$, a letter $ b \in B$, a word $w \in B^*$, a non-erasing morphism $\phi: B^{*} \to B^{*}$ and a coding function $\psi: B \to A$, such that $\phi(b)=bw$ and $(a_n)_{n \in \mathbb{N}}= \lim_{i \to \infty}\psi(\phi^i(b))$. For any positive integer $k$, let $N_k=|\phi^k(b)|$, then the first part is a direct consequence of the first part of Proposition \ref{growth}. For the second part, it is enough to note that for any $a \in A$ there exist $b_1, b_2,\cdots b_t \in B$ such that for any $k \geq 1$,
$$|a[0,N_k]|_a=\sum_{i=1}^t|\psi(\phi^i(b))|_{b_i}.$$
Then it follows the second part of Proposition \ref{growth}.
\end{proof}

\begin{proof}[Proof of Theorem \ref{morphic}]
With the same notation as previous, let $0 \leq N_1 \leq N_2 \leq \cdots \leq N_k \leq \cdots$ be defined as in Corollary \ref{sequence}. It is known that, 
$$N_k \sim G(Tk)^l\alpha^{Tk},$$
and that for any $a$ in the sequence,
$$|a[0, N_k]|_a\sim  G'(Tk)^m\sum_{i=1}^s\beta_i^{Tk},$$ where $|\beta_1|= |\beta_2|= \cdots= |\beta_s|=\beta \leq \alpha$.\\

If $\beta < \alpha$, then there exists $c' <0$, such that $$|a[0, N_k]|_a=O(N_k^{c'}) \nsim C\frac{N_k}{(\log(N_k))^{\gamma}}.$$

If $\beta=\alpha=1$, then, from Corollary \ref{sequence}, $|a[0, N_k]|_a \sim G'(Tk)^m$ and $N_k\sim G'(Tk)^l$. In this case
$$ G'(Tk)^m \nsim C\frac{G(Tk)^l}{(\log(G(Tk)^l))^{\gamma}},$$
for any $C$, $m$ and $l$.

If $\beta = \alpha >1$, then, from Corollary \ref{sequence}, $|a[0, N_k]|_a \sim G'(Tk)^m(\alpha)^{Tk}$ and $N_k\sim G'(Tk)^l(\alpha)^{Tk}$.
 $$(\log{N_k})^{\gamma} \sim (\log{GT^l}+l\log{k}+kT\log{\alpha})^{\gamma}\sim C(k)^{\gamma},$$for some constant $C$. However, since $0 < \gamma <1$, $$C\frac{N_k}{(\log{N_k})^{\gamma}} \sim C'(k)^{l-\gamma}\alpha^{kT} \nsim G'(kT)^m\alpha^{kT},$$ for any integer $m$.

\end{proof}

\begin{prop}\label{density}[\cite{Landau08}]
For any non-negative integer $N$, let $B(N)$ be the number of integers $n$ such that $0 \leq n \leq N$ and $s_2(n)=1$. Then, there exists a constant K (known as Landau-Ramanujan constant), such that $$B(N)\sim K\frac{N}{\sqrt{\log(N)}}.$$
\end{prop}

\begin{thm}
The sequence $(s_2(n))_{n \in \mathbb{N}}$ is not morphic.
\end{thm}

\begin{proof}
It is a direct consequence if Theorem \ref{morphic} and Proposition \ref{density}.
\end{proof}

\begin{thm}
Let $(s'_2(n))_{n \in \mathbb{N}}$ be a $0,1$ sequence such that, for any natural number $n$, $s'_2(1)=1$ if and only if $n$ is a sum of two {\em non-zero} squares. Then $(s'_2(n))_{n \in \mathbb{N}}$ is not morphic.
\end{thm}

\begin{proof}
Similarly for any natural number $N$, define $B'(N)$ to be the number of integers $n$ such that $o \leq n \leq N$ ans $s'_2(n)=1$. Obviously,
$$|B(N)-B'(N)| \leq \sqrt{N} =o(K\frac{N}{\log(N)}),$$
with the same notation in Proposition \ref{density}.
Thus, $$B'(N)\sim K\frac{N}{\sqrt{\log(N)}},$$
and on can conclude using Theorem \ref{morphic}.
\end{proof}

\section {Discussion}

The key technique of this proof is base on the estimation of $B(N)$, which is obtained in using the fact that $(s_2(n))_{n \in \mathbb{N}}$ is multiplicative. For any pair of integers $r,t$, let $(s_{r,t}(n))_{n \in \mathbb{N}}$ be the sequence characterizing the integers that are the sum of r numbers of $k$-power. The technique applied in \cite{Landau08} may not useful to estimate the number of integers $n$ satisfying $n \leq N$ and $s_{r,t}(n)=1$.\\

The author believes the following generalization of Theorem \ref{morphic} is true:
\begin{conj}
Let $(a_n)_{n \in \mathbb{N}}$ be a morphic sequence over the alphabet $A$. If there exists $a \in A$, two positive numbers $0 < \gamma < 1$, $C >0$ and an increasing sequence of integers $N_1<N_2<\cdots <N_k<\cdots$ such that $$|a[0,N_k]|_a \sim C\frac{N_k}{(\log(N_k))^{\gamma}},$$ when $k$ is large, then $(a_n)_{n \in \mathbb{N}}$ is not morphic.  
\end{conj}